\documentclass[reqno]{amsart}
\usepackage{amsmath,amsfonts,amssymb}
\usepackage{color}

\usepackage{enumerate}

\begin{document}

\title[Phase-isometries between normed spaces]{Phase-isometries between normed spaces}

\author{Dijana Ili\v{s}evi\'{c}}

\address{Department of Mathematics, Faculty of Science, University of Zagreb, Croatia}

\email{ilisevic@math.hr}

\author{Matja\v{z} Omladi\v{c}}

\address{Institute of Mathematics, Physics and Mechanics, Jadranska 19, 1000 Ljubljana, Slovenia}

\email{matjaz@omladic.net}

\author{Aleksej Turn\v{s}ek}

\address{Faculty  of Maritime Studies and Transport, University of Ljubljana, Pot pomor\-\v{s}\v{c}akov 4, 6320 Portoro\v{z}, Slovenia and Institute of Mathematics, Physics and Mechanics, Jadranska 19, 1000 Ljubljana, Slovenia}

\email{aleksej.turnsek@fpp.uni-lj.si}
\thanks{Dijana Ili\v{s}evi\'{c} has been fully supported by the Croatian Science Foundation [project number
	IP-2016-06-1046]. 
	 Matja\v{z} Omladi\v{c} was supported in part by the Ministry of Science
	and Education of Slovenia, grant P1-0222.
	Aleksej Turn\v{s}ek was supported in part by the Ministry of Science
and Education of Slovenia, grants J1-8133 and P1-0222.}





\begin{abstract}
Let $X$ and $Y$ be real normed spaces and $f \colon X\to Y$ a surjective mapping. Then $f$ satisfies  $\{\|f(x)+f(y)\|,\|f(x)-f(y)\|\}=\{\|x+y\|,\|x-y\|\}$, $x,y\in X$, if and only if $f$ is phase equivalent to a surjective linear isometry, that is, $f=\sigma U$, where $U \colon X\to Y$ is a surjective linear isometry and $\sigma \colon X\to \{-1,1\}$. This is a Wigner's type result for real normed spaces.
\end{abstract}

\maketitle

\newtheorem{theorem}{Theorem}[section]
\newtheorem{proposition}[theorem]{Proposition}
\newtheorem{lemma}[theorem]{Lemma}
\newtheorem{corollary}[theorem]{Corollary}
\newtheorem{definition}{Definition}
\theoremstyle{definition}
\newtheorem{example}[theorem]{Example}
\newtheorem{xca}[theorem]{Exercise}
\newtheorem{question}{Question}

\theoremstyle{remark}
\newtheorem{remark}[theorem]{Remark}

\section{Introduction}

Let $(H,(\cdot,\cdot))$ and $(K,(\cdot,\cdot))$ be real or complex inner product spaces and let $f \colon H \to K$ be a mapping.
Then $f$ satisfies
\begin{equation}\label{Wigner}
| (f(x),f(y))|=|(x,y)|,\quad x,y\in H,
\end{equation}
if and only $f$ is phase equivalent to a linear or an anti-linear isometry, say $U$, that is,
\begin{equation}\label{phase}
f(x)=\sigma(x) Ux,\quad x\in H,
\end{equation}
where a so-called phase function $\sigma$ takes values in modulus one scalars.
This is one of many forms of Wigner's theorem, also known as Wigner's unitary-antiunitary theorem.
It has played an important role in mathematical foundations of quantum mechanics.
More details can be found in excellent survey \cite{Ch}, see also \cite{Geher, Gyory, Molnar, Ratz}.

If $H$ and $K$ are real then it is easy to verify that \eqref{Wigner} implies
		\begin{equation}\label{main}
		\{\|f(x)+f(y)\|,\|f(x)-f(y)\|\}=\{\|x+y\|,\|x-y\|\}, \quad x,y\in H.
		\end{equation}
		In \cite[Theorem 2]{Maksa} Maksa and P\'{a}les proved the converse, that is, \eqref{main} implies \eqref{Wigner}. Thus solutions of \eqref{main}, which may be called phase-isometries, are exactly of the form \eqref{phase}. If $X$ and $Y$ are real normed spaces, then it is easy to see that any mapping $f \colon X\to Y$ of the form \eqref{phase} satisfies \eqref{main}. Therefore, it is natural to ask (see \cite[Problem 1]{Maksa})
				under what conditions,  when $X$ and $Y$ are real normed but not necessarily inner product spaces, solutions $f \colon X \to Y$ of \eqref{main} have the form \eqref{phase}?

		There are several recent papers dealing with this problem. See \cite{Huang1, Jia, Li,Zeng} for $L_p$ spaces, Tsirelson space $T$ and $l^p(\Gamma,H)$ spaces. There are also more general results.
		It has been proven in \cite[Theorem 2.8]{Huang2} that \eqref{main} implies \eqref{phase} if $f$ is surjective and $X$ is smooth,
		in \cite[Theorems 9 and 11]{WB} if $f$ is surjective and $\dim{X}=2$, or $f$ is surjective and $X$ is strictly convex. In \cite[Theorem 2.4]{IT} the same implication has been proven without the assumption of surjectivity, assuming only that $Y$ is strictly convex. The aim of this note is to prove the general case, that is, we give a positive answer to the question of Maksa and P\'{a}les. Namely, we prove that \eqref{main} implies \eqref{phase} for surjective mapping $f$, without any assumptions on $X$ and $Y$. Let us mention that the assumption of surjectivity cannot be omitted since e.g.~$f \colon \mathbb{R}\to\mathbb{R}^2$, where $\mathbb{R}^2$ is endowed with the norm $\|(x,y)\|=\max\{|x|,|y|\}$, defined by $f(t)=(t,\sin t)$ satisfies \eqref{main} but it does not satisfy \eqref{phase}.

\section{Preliminaries}
The following result is in fact \cite[Lemma 2.1 and Remark 3.5]{Huang2}.
		We give the proofs for the sake of completeness.
\begin{lemma}\label{lemma preliminaries}
Let X and Y be real normed spaces, and let a surjective $f \colon X \to Y$ satisfy \eqref{main}.
\begin{itemize}
\item[(i)]
Then f is a norm preserving map, injective, and $f(-x) = -f(x)$ for all $x\in X.$
\item[(ii)]
Let $\{x_n\}$ be any sequence in $X$ converging to $x \in X$. Then there is a subsequence $\{x_{n_i}\}$ such that $\{f(x_{n_i})\}$ converges to $f(x)$ or $-f(x)$.
\end{itemize}
\end{lemma}

\begin{proof}
	
(i): For $y=x$ we have $\{2\|f(x)\|,0\}=\{2\|x\|,0\}$, therefore $f$ is norm preserving.

Let $x \in X$ be nonzero. Since $f$ is surjective, there exists $y \in X$ such that $f(y)=-f(x)$.
Then we have $\{0,2\|f(x)\|\}=\{\|x+y\|,\|x-y\|\}$, hence $y=x$ or $y=-x$.
If $y=x$ then $f(x)=f(y)=-f(x)$, which first implies $f(x)=0$, and then,
since $f$ is norm preserving,
 $x=0$, which contradicts the assumption $x \neq 0$.
Therefore, $y=-x$.
In other words, $f(-x)=-f(x)$ for all $x \in X$.

If $x,y \in X$ are such that $f(x)=f(y)$ then we have $\{2\|f(x)\|,0\}=\{\|x+y\|,\|x-y\|\}$.
Then $y=-x$ or $y=x$.
If $y=-x \neq 0$ then $f(x)=f(y)=f(-x)=-f(x)$, therefore $f(x)=0$, which implies $x=0$; a contradiction.
Hence, $f$ is injective.

(ii): Let $\varepsilon > 0$.
Let $I=\{x_n \, : \, \Vert f(x_n)-f(x) \Vert < \varepsilon\}$ and $J=\{x_n \, : \, \Vert f(x_n)+f(x) \Vert < \varepsilon\}$.
There exists $n_0 \in \mathbb{N}$ such that for every $n > n_0$ we have $\Vert x_n-x \Vert < \varepsilon$.
Then for every $n > n_0$ we also have $\Vert f(x_n)-f(x) \Vert < \varepsilon$ or  $\Vert f(x_n)+f(x) \Vert < \varepsilon$, that is,
$x_n \in I$ or $x_n \in J$.
At least one of these sets, $I$ and $J$, is infinite.
Suppose that $I$ is infinite.
Let $p \colon \mathbb{N} \to \mathbb{N}$ be strictly increasing such that $x \circ p \colon \mathbb{N} \to I$.
Set $n_i=p(i)$.
Then $\{x_{n_i}\}$ is a subsequence of $\{x_n\}$ such that $\{f(x_{n_i})\}$ converges to $f(x)$.
If $J$ is infinite we analogously conclude that there is a subsequence $\{x_{n_i}\}$ such that $\{f(x_{n_i})\}$ converges to $-f(x)$.
\end{proof}

\section{ The elaboration of specifics }
In what follows $f$ will always denote a surjective map from a real normed space $X$ to a real normed space $Y$ satisfying \eqref{main}.
For $x,y\in S$, where $S \in \{X,Y\}$, let
$$H_{x,y}=\{u\in S: \|u-x\|=\|u-y\|=\tfrac{1}{2}\|x-y\|\}.$$
Then $\frac{1}{2}(x+y)\in H_{x,y}$,  $H_{-x,-y}=-H_{x,y}$, and $H_{-x,y}=-H_{x,-y}$.
Let $x,y \in X$ and choose
$$u\in H_{x,y}\cup H_{-x,-y}\cup H_{-x,y}\cup H_{x,-y}.$$
Then for $f(u)$, $f(x)$ and $f(y)$ we have eight possibilities:\\[2mm]
(a) $\|f(u)-f(x)\|=\|f(u)-f(y)\|=\frac{1}{2}\|f(x)-f(y)\|.$\\
(b) $\|f(u)-f(x)\|=\|f(u)-f(y)\|=\frac{1}{2}\|f(x)+f(y)\|.$\\
(c) $\|f(u)+f(x)\|=\|f(u)-f(y)\|=\frac{1}{2}\|f(x)-f(y)\|.$\\
(d) $\|f(u)+f(x)\|=\|f(u)-f(y)\|=\frac{1}{2}\|f(x)+f(y)\|.$\\
(e) $\|f(u)-f(x)\|=\|f(u)+f(y)\|=\frac{1}{2}\|f(x)-f(y)\|.$\\
(f)  $\|f(u)-f(x)\|=\|f(u)+f(y)\|=\frac{1}{2}\|f(x)+f(y)\|.$\\
(g) $\|f(u)+f(x)\|=\|f(u)+f(y)\|=\frac{1}{2}\|f(x)-f(y)\|.$\\
(h) $\|f(u)+f(x)\|=\|f(u)+f(y)\|=\frac{1}{2}\|f(x)+f(y)\|.$\\

\begin{remark}\label{allprop}
	
		Let $x,y \in X$.
		
(i)
If $u \in H_{x,y} \cup H_{-x,-y}$ then (a), (c), (e) and (g) can happen if $\Vert f(x)-f(y) \Vert = \Vert x-y \Vert$,
and (b), (d), (f) and (h) can happen if $\Vert f(x)-f(y) \Vert = \Vert x+y \Vert$.

(ii)
Let $u \in H_{x,y} \cup H_{-x,-y}$.
If (a) holds then $f(u) \in H_{f(x), f(y)}$, if (d) holds then $f(u) \in H_{-f(x), f(y)}$, if (f) holds then $f(u) \in H_{f(x), -f(y)}$, and
if (g) holds then $f(u) \in H_{-f(x), -f(y)}$.

(iii)
By the triangle inequality, in cases (b) and (h) we get $\Vert f(x)-f(y) \Vert \leq \Vert f(x)+f(y) \Vert$. Indeed,
$$\|f(x)-f(y)\|=\|(f(u) \pm f(x))-(f(u) \pm f(y))\|\leq\|f(x)+f(y)\|.$$
Similarly, in cases (c) and (e) we get
$\Vert f(x)+f(y) \Vert \leq \Vert f(x)-f(y) \Vert$.
\end{remark}

\begin{proposition}\label{sets1}
Let $x,y\in X$.
\begin{itemize}
\item[(i)] If $\|x-y\|<\|x+y\|$ and $\|f(x)-f(y)\|=\|x-y\|$, then
$$f(H_{x,y}\cup H_{-x,-y})=H_{f(x),f(y)}\cup H_{-f(x),-f(y)}.$$
\item[(ii)] If $\|x-y\|<\|x+y\|$ and $\|f(x)-f(y)\|=\|x+y\|$, then
$$f(H_{x,y}\cup H_{-x,-y})=H_{-f(x),f(y)}\cup H_{f(x),-f(y)}.$$
\item[(iii)]  $f(H_{x,0} \cup H_{-x,0})=H_{f(x),0} \cup H_{-f(x),0}$.
\end{itemize}
\end{proposition}
\begin{proof}
	
(i): By Remark \ref{allprop}(i), we have that (a), (c), (e) or (g) holds. However, Remark \ref{allprop}(iii) implies that (c) and (e) cannot hold. Thus, only (a) and (g) can happen.
Then for  $u \in H_{x,y} \cup H_{-x,-y}$ we have $f(u) \in H_{f(x),f(y)}\cup H_{-f(x),-f(y)}$ by Remark \ref{allprop}(ii).
Therefore, $f(H_{x,y}\cup H_{-x,-y}) \subseteq H_{f(x),f(y)}\cup H_{-f(x),-f(y)}$.
Since $f$ is bijective, we may and do replace $f$ in these considerations by $f^{-1}$ to get the reverse inclusion.

(ii): Analogously as in (i) we conclude that only (d) and (f) can happen. Hence, $f(H_{x,y}\cup H_{-x,-y}) \subseteq H_{-f(x),f(y)}\cup H_{f(x),-f(y)}$.
The desired conclusion follows after applying this inclusion to $f^{-1}$ as above.

(iii): If $y=0$ then eight cases (a)--(h) are reduced to two, thus
		$f(H_{x,0}\cup H_{-x,0})\subseteq H_{f(x),0}\cup H_{-f(x),0}$.
		Replacing $x$ with $f(x)$ and $f$ with $f^{-1}$ we get the reversed inclusion.
\end{proof}



\medskip

Let $x,y\in X$ and  let  $G_1=H_{x,y}\cup H_{-x,-y}$.
For $n=2,3,\ldots$, define
$$G_n=\{u\in G_{n-1}: \min\{\|u-v\|,\|u+v\|\}\leq \tfrac{1}{2}d_{n-1}\; \text{for all}\; v\in G_{n-1}\}.$$
Here
$$d_{n-1}=\sup_{u,v\in G_{n-1}}\min\{\|u-v\|,\|u+v\|\},\quad n=2,3,\ldots.$$
Clearly $d_1$ is finite and from $d_n\leq\frac{1}{2}d_{n-1}$ it follows that $d_n\to0$ when $n\to\infty$. Hence the intersection of the nested sets $G_1\supseteq G_2\supseteq\ldots$ is either zero or consists of exactly two points, say $z$ and $-z$.

\begin{proposition}\label{center}
Let $x,y\in X$ be such that $\|u+v\|\geq \|u-v\|$ for all $u,v\in H_{x,y}$. Then the intersection of the sets $G_n$ is $\{z,-z\}$, where $z=\frac{1}{2}(x+y)$.
\end{proposition}
\begin{proof}
Since $u\in H_{x,y}$ if and only if $-u\in -H_{x,y}=H_{-x,-y}$
we conclude that
\begin{equation}\label{eq min}
\min\{\|u-v\|,\|u+v\|\}=
\begin{cases}
\|u-v\| & \text{if}\quad u,v \in H_{x,y} \,\text{ or }\, u,v\in -H_{x,y}\\
\|u+v\| & \text{if}\quad u\in H_{x,y}, v\in -H_{x,y}\, \text{ or }\, u\in -H_{x,y}, v\in H_{x,y}.
\end{cases}
\end{equation}
For any $u\in H_{x,y}$ let $\tilde{u}=x+y-u$. For any $u\in -H_{x,y}$ let $\overline{u}=-x-y-u$. If $u\in H_{x,y}$, then  $\tilde{u}=x+y-u\in H_{x,y}$ as well. Indeed, this follows from
$$\tilde{u}-x=y-u\quad\text{and}\quad \tilde{u}-y=x-u.$$
Thus, $u \in G_1 \cap H_{x,y}=H_{x,y}$ implies $\tilde{u} \in G_1 \cap H_{x,y}=H_{x,y}$.
Similarly, from the fact that $u\in -H_{x,y}$ implies $\overline{u}\in -H_{x,y}$, we infer that $\overline{u}\in G_1\cap -H_{x,y}$ whenever $u\in G_1\cap -H_{x,y}$.

Assume {inductively} that $\tilde{u}\in G_{n-1}\cap H_{x,y}$ whenever $u\in G_{n-1}\cap H_{x,y}$ and let $u\in G_n\cap H_{x,y}$. If $v\in G_{n-1}\cap H_{x,y}$, {we conclude from \eqref{eq min} that}
\begin{align*}
\min \{\|\tilde{u}-v\|,\|\tilde{u}+v\|\}&=\|\tilde{u}-v\|=\|x+y-u-v\|=\|\tilde{v}-u\|\\
&=\min\{\|\tilde{v}-u\|,\|\tilde{v}+u\|\}\leq\tfrac{1}{2}d_{n-1}.
\end{align*}
In case that $v\in G_{n-1}\cap -H_{x,y}$, we use \eqref{eq min} again to see that
\begin{align*}
\min \{\|\tilde{u}-v\|,\|\tilde{u}+v\|\}&=\|\tilde{u}+v\|=\|x+y-u+v\|=\|-\overline{v}-u\|\\
&=\min\{\|\overline{v}+u\|,\|\overline{v}-u\|\}\leq\tfrac{1}{2}d_{n-1}.
\end{align*}
Therefore $\tilde{u}\in G_n\cap H_{x,y}$ as well.
In the same way we show that $u\in G_n\cap -H_{x,y}$ implies $\overline{u}\in G_n\cap -H_{x,y}$ for every positive integer $n$.

Next we show by induction that $z=\frac{1}{2}(x+y)\in G_n\cap H_{x,y}$ for each $n$. First we see that $z\in G_1\cap H_{x,y}$ since $z-x=\frac{1}{2}(y-x)$ and $z-y=\frac{1}{2}(x-y)$. Assume that $z\in G_{n-1}\cap H_{x,y}$ and $u\in G_{n-1}$. If $u\in G_{n-1}\cap H_{x,y}$, then $\tilde{u}\in G_{n-1}\cap H_{x,y}$ by what we proved earlier and by \eqref{eq min} we have
\begin{align*}
\min \{\|z-u\|,\|z+u\|\}&=\|z-u\|=\tfrac{1}{2}\|x+y-2u\|=\tfrac{1}{2}\|\tilde{u}-u\|\\
&=\tfrac{1}{2}\min\{\|\tilde{u}-u\|,\|\tilde{u}+u\|\}\leq\tfrac{1}{2}d_{n-1}.
\end{align*}
If $u\in G_{n-1}\cap -H_{x,y}$, then $\overline{u}\in G_{n-1}\cap -H_{x,y}$ and again using \eqref{eq min} we get
\begin{align*}
\min \{\|z-u\|,\|z+u\|\}&=\|z+u\|=\tfrac{1}{2}\|x+y+2u\|=\tfrac{1}{2}\|u-\overline{u}\|\\
&=\tfrac{1}{2}\min\{\|u-\overline{u}\|,\|u+\overline{u}\|\}\leq\tfrac{1}{2}d_{n-1}.
\end{align*}
Hence $z\in G_n\cap H_{x,y}$. The conclusion is that $z\in\cap_1^\infty G_n$ and then also $-z\in\cap_1^\infty G_n$. The proof is complete.
\end{proof}

		\begin{lemma}\label{important}
		Let $x,y\in X$.
			\begin{itemize}
				\item[(i)] If $x$ is nonzero and $0\leq a<b$ are real numbers, then for $s=ax$ and $t=bx$ we have $\|u+v\|\geq\|u-v\|$ for all $u,v\in H_{s,t}$.
				\item[(ii)] Let $x,y$ be such that $\|x+y\|\geq 2\|x-y\|$. Then $\|u+v\|\geq\|u-v\|$ for all $u,v\in H_{x,y}$.
			\end{itemize}
		\end{lemma}
	
		\begin{proof}
			(i): Let $\varphi$ be a support functional at $x$ (a norm-one linear functional in $X^*$ such that $\varphi(x)=\Vert x \Vert$).
			So, for $s=ax$ and $t=bx$  we have $\varphi(s)=\|s\|$, $\varphi(t)=\|t\|$ and $\varphi(s+t)=\|s+t\|$. Take $u\in H_{s,t}$. Then $\varphi(u-s) \leq\|u-s\|=\frac{1}{2}\|s-t\|$ and $\varphi(t-u) \leq\|t-u\|=\frac{1}{2}\|s-t\|$. But $\varphi(u-s)+\varphi(t-u)=\varphi(t-s)=\|t-s\|$, hence $\varphi(u-s)=\varphi(t-u)$. From this we get $\varphi(u)=\frac{1}{2}\varphi(s+t)=\frac{1}{2}\|s+t\|$ for all $u\in H_{s,t}$.
			If $u,v\in H_{s,t}$, then
			$$\|u+v\|\geq\varphi(u+v)=\varphi(u)+\varphi(v)=\varphi(s+t)=\|s+t\|.$$
			On the other hand
			$$\|u-v\|=\|(u-s)-(v-s)\|\leq\|u-s\|+\|v-s\|=\|s-t\|.$$
			Since $\|s-t\|\leq\|s+t\|$ we conclude that $\|u+v\|\geq\|u-v\|$ for all $u,v\in H_{s,t}$.
			
			(ii): Let $u,v\in H_{x,y}$. Then $x-u, \, x-v, \, u-y, \, v-y\in H_{x-y,0}$. From (i) (take $s=0$ and $t=x-y$) we conclude that
			\begin{align*}
			\|x-y\|&=\tfrac{1}{2}\|x-y\|+\tfrac{1}{2}\|x-y\|=\|x-u\|+\|v-y\|\geq\|(x-u)+(v-y)\|\\
			&\geq\|(x-u)-(v-y)\|=\|(x+y)-(u+v)\|\geq\|x+y\|-\|u+v\|.
			\end{align*}
			Hence $\|u+v\|\geq\|x+y\|-\|x-y\|\geq\|x-y\|$. Since $\|u-v\|\leq\|x-y\|$, the claim follows.
		\end{proof}

\begin{proposition}\label{prop center}
Let $x,y\in X$ and suppose that $\|x-y\|<\|x+y\|$.
\begin{itemize}
	\item[(i)] If  $\Vert f(x)-f(y) \Vert = \Vert x-y \Vert$, $\Vert u+v \Vert \geq \Vert u-v \Vert$ for all $u, v \in H_{x,y}$ and $\|u'+v'\|\geq\|u'-v'\|$ for all  $u',v' \in  H_{f(x),f(y)}$, then $f(\frac{1}{2}(x+y))=\pm \frac{1}{2}(f(x)+f(y))$. 
	\item[(ii)] If  $\Vert f(x)-f(y) \Vert = \Vert x+y \Vert$, $\Vert u+v \Vert \geq \Vert u-v \Vert$ for all $u, v \in H_{x,y}$ and $\|u'+v'\|\geq\|u'-v'\|$ for all $u', v' \in  H_{f(x),-f(y)}$, then $f(\frac{1}{2}(x+y))=\pm \frac{1}{2}(f(x)- f(y))$. 
	\item[(iii)] $f\left(\tfrac{1}{2}x\right)= \pm \tfrac{1}{2}f(x)$.
\end{itemize} 
\end{proposition}

\begin{proof}
	(i): By Proposition \ref{sets1}(i) we have
$$f(H_{x,y}\cup -H_{x,y})=H_{f(x),f(y)}\cup -H_{f(x),f(y)}.$$ 
Let 
$$G_1=H_{x,y}\cup -H_{x,y},\quad G_1'=H_{f(x),f(y)}\cup -H_{f(x),f(y)},$$
and for $n=2,3,\ldots$ let
$$d_{n-1}=\sup_{u,v\in G_{n-1}}\min\{\|u-v\|,\|u+v\|\},\,d_{n-1}'=\sup_{u',v'\in G_{n-1}'}\min\{\|u'-v'\|,\|u'+v'\|\},$$
where
$$G_n=\{u\in G_{n-1}: \min\{\|u-v\|,\|u+v\|\}\leq \tfrac{1}{2}d_{n-1}\; \text{for all}\; v\in G_{n-1}\},$$
$$G_n'=\{u'\in G_{n-1}': \min\{\|u'-v'\|,\|u'+v'\|\}\leq \tfrac{1}{2}d_{n-1}'\; \text{for all}\; v'\in G_{n-1}'\}.$$

From \eqref{main} and from $f(G_1)=G_1'$ we get
$d_1=d_1'$.
If we assume inductively that $f(G_{n-1})=G_{n-1}'$ and $d_{n-1}=d_{n-1}'$, then for $u\in G_n$ and $v'\in G_{n-1}'$ we have $v'=f(v)$ for some $v\in G_{n-1}$ while
\begin{align*}
\min\{\|f(u)-v'\|,\|f(u)+v'\|\} & =\min\{\|f(u)-f(v)\|,\|f(u)+f(v)\|\}\\
& = \min\{\|u-v\|,\|u+v\|\}\leq\tfrac{1}{2}d_{n-1}=\tfrac{1}{2}d_{n-1}'.
\end{align*}
Therefore, $f(G_n)\subseteq G_n'$. Similarly, from $G_{n-1}=f^{-1}(G_{n-1}')$ and $d_{n-1}=d_{n-1}'$ we get $f^{-1}(G_n')\subseteq G_n$. Hence $f(G_n)=G_n'$ for every positive integer $n$. By Proposition \ref{center} we conclude that $f(\frac{1}{2}(x+y))=\pm\frac{1}{2}(f(x)+f(y))$.

(ii): By Proposition \ref{sets1}(ii),
$$f(H_{x,y}\cup H_{-x,-y})=H_{f(x),f(-y)}\cup H_{f(-x),f(y)}.$$
We repeat the steps from (i) to get $f(\frac{1}{2}(x+y))=\pm\frac{1}{2}(f(x)-f(y))$.

(iii) We first apply Lemma \ref{important}(i) for $a=0$ and $b=1$ to conclude $\Vert u+v \Vert \geq \Vert u-v \Vert$ for all $u,v \in H_{x,0}$ and all $u', v'\in H_{f(x),0}$.
As in (i) we get $f(G_n)=G_n'$ for every positive integer $n$ and then we apply Proposition \ref{center}.
\end{proof}

		\begin{remark}\label{dod}
If nonzero $x,y \in X$ are such that $\Vert x+y \Vert \geq 2 \Vert x-y \Vert$ then either the assumptions of (i) or the assumptions of (ii) in Proposition \ref{prop center} are satisfied. Namely, Lemma \ref{important}(ii) implies $\Vert u+v \Vert \geq \Vert u-v \Vert$ for all $u,v \in H_{x,y}$.
Furthermore, if $\Vert f(x)-f(y) \Vert = \Vert x-y \Vert$ we apply Lemma \ref{important}(ii) to $f(x)$ and $f(y)$ instead of $x$ and $y$, and if $\Vert f(x)-f(y) \Vert = \Vert x+y \Vert$ we apply Lemma \ref{important}(ii) to $f(x)$ and $-f(y)$. Hence $f(\frac{1}{2}(x+y))=\frac{1}{2}(\pm f(x) \pm f(y))$.
\end{remark}


\begin{proposition}\label{hom}
Let $x\in X$ and {let} $\lambda\in\mathbb{R}$. Then $f(\lambda x)=\pm\lambda f(x)$.
\end{proposition}

\begin{proof}
From Proposition \ref{prop center}(iii) it follows that $f(2x)=\pm 2 f(x)$.
We proceed by induction.
Suppose that $f(nx)=\pm n f(x)$.

If $f(nx)=n f(x)$ then Proposition \ref{prop center}(i) and Lemma \ref{important}(i) imply
$$f(\tfrac{1}{2}(x+nx))= \pm \tfrac{1}{2}(f(x)+f(nx))= \pm \tfrac{1}{2}(n+1)f(x),$$
hence
$$(n+1)f(x) = \pm 2 f(\tfrac{1}{2}(n+1)x)=\pm f((n+1)x).$$

If $f(nx)=-n f(x)$ then we use Proposition \ref{prop center}(ii) and Lemma \ref{important}(i) to conclude
$$f(\tfrac{1}{2}((n+1)x))= \pm \tfrac{1}{2}(f(x)-f(nx))= \pm \tfrac{1}{2}(n+1)f(x).$$
As before we get $f((n+1)x)=\pm (n+1)f(x)$.

Thus by induction $f(nx)=\pm n f(x)$ for every positive integer $n$. 
Then on one hand we have $f(mx)=\pm m f(x)$ and on the other
$$f(mx)=f(n \cdot \tfrac{m}{n}x)= \pm n f(\tfrac{m}{n}x),$$
which implies $f(\tfrac{m}{n}x)=\pm \tfrac{m}{n} f(x)$ for all positive integers $m$ and $n$.

We apply Lemma \ref{lemma preliminaries}(ii) to conclude that $f(\lambda x)=\pm\lambda f(x)$ for all real $\lambda\geq0$. Since $f$ is odd this holds for all $\lambda\in\mathbb{R}$ and the proof is complete.
\end{proof}

For $x,y \in X$ we shall write $x\perp y$ if $x$ is Birkhoff--James orthogonal to $y$, that is,
$$\|x+\lambda y\|\geq\|x\|\quad \text{for all }\lambda\in\mathbb{R}.$$
Note that this notion of orthogonality is homogeneous, but is neither symmetric nor additive.

\begin{lemma}\label{lin ind}
Let $x,y\in X$. Then the following holds:
\begin{itemize}
\item[(i)] If $x$ and $y$ are linearly independent, then $f(x)$ and $f(y)$ are linearly independent.
\item[(ii)] If $x\perp y$, then $f(x)\perp f(y)$.
\end{itemize}
\end{lemma}
\begin{proof}
(i): Suppose $x$ and $y$ are linearly independent and let $\lambda\in\mathbb{R}$. From
\begin{equation}\label{eqeq}
\{\|f(x)+f(\lambda y)\|,\|f(x)-f(\lambda y)\|\}=\{\|x+\lambda y\|,\|x-\lambda y\|\}
\end{equation}
it follows that
\begin{equation}\label{linearly independent}
\|f(x)+f(\lambda y)\|\ne0\quad\text{and}\quad \|f(x)-f(\lambda y)\|\ne0.
\end{equation}
From $f(\lambda y)=\pm\lambda f(y)$ and (\ref{linearly independent}) the claim follows.

(ii): Suppose $x\perp y$ and let $\lambda\in\mathbb{R}$. Then $\min\{\|x+\lambda y\|,\|x-\lambda y\|\}\geq\|x\|$.  From (\ref{eqeq}) it follows that
\begin{equation}\label{preserving orth}
\|f(x)+f(\lambda y)\|\geq\|x\|=\|f(x)\|\quad\text{and}\quad \|f(x)-f(\lambda y)\|\geq\|x\|=\|f(x)\|.
\end{equation}
Now the claim follows from $f(\lambda y)=\pm\lambda f(y)$ and (\ref{preserving orth}).
\end{proof}

If $M$ is a set, then $\langle M\rangle$ will denote the subspace generated by the set $M$.
If $M=\{x\}$ is a singleton, we write $\langle x \rangle$ for the one dimensional subspace generated by $M$.
Analogously, if $M=\{x, y\}$, we write $\langle x, y \rangle$ for the two dimensional subspace generated by $M$.
Furthermore, $[x,y]$ denotes $\{(1-\lambda)x+\lambda y \, : \, \lambda \in [0,1]\}$.

\begin{lemma}\label{very important}
Let $x,y\in X$ be linearly independent vectors such that $\Vert x \Vert \leq \Vert y \Vert$.
Let $z=\frac{1}{4}x+\frac{3}{4}y$.
If $f([z,y]) \subseteq \langle f(z),f(y)\rangle$ then $f([x,y]) \subseteq \langle f(z), f(y) \rangle = \langle f(x), f(y) \rangle$.
\end{lemma}

\begin{proof}
Let $\varphi \colon \mathbb{R} \to \mathbb{R}^+$ be defined by $\varphi(\lambda)=\|(1-\lambda)x+\lambda y\|$.
Since $\varphi$ is continuous, convex and $\varphi(\lambda) \to \infty$ when $\vert \lambda \vert \to \infty$, there is (at least one) $\lambda_0$ such that
$\min_\lambda\|(1-\lambda)x+\lambda y\|= \|(1-\lambda_0)x+\lambda_0 y\|=:d$.
Since $x$ and $y$ are linearly independent, $d>0$.
Take $\mu \in \mathbb{R}$ such that $\tfrac{3}{4} \leq \mu \leq \tfrac{d}{4\Vert y \Vert} + \tfrac{3}{4}$.
Note that $d \leq \Vert y \Vert$, so $\tfrac{d}{4\Vert y \Vert} + \tfrac{3}{4} \leq 1$.
For $u=(1-\mu)x+\mu y$ and $\tilde{u}=2z-u$, the reflection of $u$ over $z$, we get
$$\Vert u - \tilde{u} \Vert = 2 \Vert z-u \Vert = 2(\mu-\tfrac{3}{4}) \Vert x-y \Vert 
\leq 4(\mu-\tfrac{3}{4}) \Vert y \Vert \leq  d.$$
On the other hand,
$$\Vert u + \tilde{u} \Vert = 2 \Vert z \Vert \geq 2d,$$
hence $\Vert u + \tilde{u} \Vert \geq 2 \Vert u - \tilde{u} \Vert$.
By Proposition \ref{prop center}, see Remark \ref{dod}, we conclude that
$$f(z) = f\left(\tfrac{1}{2}(u+\tilde{u})\right)= \pm \tfrac{1}{2}f(u) \pm \tfrac{1}{2}f(\tilde{u}).$$
Since $u \in [z,y]$, by our assumption we have $f(u) \in \langle f(z), f(y) \rangle$.
Then $f(\tilde{u}) \in \langle f(z), f(y) \rangle$ as well.
Note that
$\tilde{u}=(1-\tilde{\mu})x+\tilde{\mu}y$ with $\tfrac{3}{4}-\tfrac{d}{4\Vert y \Vert} \leq \tilde{\mu} \leq \tfrac{3}{4}$.
If we set 
$w_1=\left(\tfrac{1}{4}+\tfrac{d}{4\Vert y \Vert}\right)x+\left(\tfrac{3}{4}-\tfrac{d}{4\Vert y \Vert}\right)y$
then we have proved that $f([w_1, y]) \subseteq \langle f(z), f(y) \rangle$.

We repeat this procedure. 
Take $\nu \in \mathbb{R}$ such that $\tfrac{3}{4}-\tfrac{d}{4\Vert y \Vert} \leq \nu \leq \tfrac{3}{4}$.
For $v=(1-\nu)x+\nu y  \in [w_1,z]$ and $\tilde{v}=2w_1-v$ we have
$$\Vert v - \tilde{v} \Vert = 2 \Vert w_1-v \Vert = 2 \left(\nu + \tfrac{d}{4\Vert y \Vert} - \tfrac{3}{4}\right)\Vert x-y \Vert \leq 4\left(\nu+\tfrac{d}{4\Vert y \Vert}- \tfrac{3}{4}\right) \Vert y \Vert \leq d$$
and $\Vert v + \tilde{v} \Vert = 2 \Vert w_1 \Vert \geq 2d$.
Proposition \ref{prop center} implies
$$f(w_1)=f\left(\tfrac{1}{2}(v+\tilde{v})\right)=\pm \tfrac{1}{2}f(v) \pm \tfrac{1}{2}f(\tilde{v}).$$
Since $f(w_1), f(v) \in \langle f(z), f(y) \rangle$ it follows that $f(\tilde{v}) \in \langle f(z), f(y) \rangle$.
Note that $\tilde{v}=(1-\tilde{\nu})x+\tilde{\nu} y$ with $\tilde{\nu}=\tfrac{3}{2}-\nu-\tfrac{d}{2\Vert y \Vert}$ 
and $\tfrac{3}{4}-\tfrac{d}{2\Vert y \Vert} \leq \tilde{\nu} \leq \tfrac{3}{4}-\tfrac{d}{4\Vert y \Vert}$.
Hence $f([w_2, y]) \subseteq \langle f(z), f(y) \rangle $ for $w_2=\left(\tfrac{1}{4}+\tfrac{d}{2\Vert y \Vert}\right)x+\left(\tfrac{3}{4}-\tfrac{d}{2\Vert y \Vert}\right)y$.

Continuing this procedure we get, for every positive integer $n$, $f([w_n, y]) \subseteq \langle f(z), f(y) \rangle $ with $w_n=\left(\tfrac{1}{4}+n \tfrac{d}{4\Vert y \Vert}\right)x+\left(\tfrac{3}{4}-n \tfrac{d}{4\Vert y \Vert}\right)y$.
There exists a positive integer $n_0$ such that that $n_0 \geq \tfrac{3\Vert y \Vert}{d}$. Then $x \in [w_{n_0}, y]$, hence $f(x) \in \langle f(z), f(y) \rangle$ and finally $f(z) \in \langle f(x), f(y) \rangle$. Therefore $f([x, y]) \subseteq f([w_{n_0}, y]) \subseteq \langle f(z), f(y) \rangle = \langle f(x), f(y) \rangle$.
\end{proof}

\begin{proposition}\label{two}
Let $\Pi\subseteq X$ be a two dimensional subspace. Then $f(\Pi)\subseteq Y$ is also a two dimensional subspace.
\end{proposition}
\begin{proof}
Let $\Pi\subseteq X$ be a two dimensional subspace, $\Pi=\langle x,y\rangle$, where $x$ and $y$ are linearly independent. With no loss of generality we may and do assume that $\|x\|\leq\|y\|$. Let $z=\frac{1}{4}x+\frac{3}{4}y\in[x,y]$. Take any two points $w_1,w_2\in[z,y]$, say
$$w_1=(1-\lambda)x+\lambda y,\quad w_2=(1-\mu)x+\mu y,$$
where $\frac{3}{4}\leq\lambda,\mu\leq 1$. Then
\begin{align*}
\|w_1+w_2\|&=\|(2-\lambda-\mu)x+(\lambda+\mu)y\|\geq(\lambda+\mu)\|y\|-(2-\lambda-\mu)\|x\|\\
&\geq(2(\lambda+\mu)-2)\|y\|\geq\|y\|
\end{align*}
and
$$\|w_1-w_2\|=\|(\mu-\lambda)(x-y)\|\leq|\mu-\lambda|(\|x\|+\|y\|)\leq\tfrac{1}{2}\|y\|,$$
hence $\|w_1+w_2\|\geq2\|w_1-w_2\|$.  In particular, from Proposition \ref{prop center} we infer   that $f(\frac{1}{2}(z+y))\in\langle f(z),f(y)\rangle$. But the same is true for the midpoints $z_1$ of the segment $[z,\frac{1}{2}(z+y)]$ and $z_2$ of the segment $[\frac{1}{2}(z+y),y]$. Now we have four segments $[z,z_1],[z_1,\frac{1}{2}(z+y)],[\frac{1}{2}(z+y),z_2],[z_2,y]$. The same reasoning implies that all midpoints of the above four segments are again in $\langle f(z),f(y)\rangle$. Continuing in this way, and using the facts that dyadic fractions are dense and $\langle f(z),f(y)\rangle$ is closed,  by Lemma \ref{lemma preliminaries} we conclude that $f(u)\in \langle f(z),f(y)\rangle$ for all $u\in[z,y]$. Now it follows by Lemma \ref{very important} that $f([x,y])\subseteq \langle f(x),f(y)\rangle=\Pi'$. 

Let us consider the segment $[x,-y]$. By repeating the above steps and using the fact that $f(-y)=-f(y)$ we get $f([x,-y])\subseteq \langle f(x),f(y)\rangle=\Pi'$ and then also $f([-x,y])\subseteq \Pi'$ and $f([-x,-y])\subseteq \Pi'$.

Take any $\xi\in\langle x,y\rangle=\Pi$ and let $\xi'$ be the intersection of the set $\{\mu\xi: \mu>0\}$ with the union of the segments $[\pm x, \pm y]$ (parallelogram with the vertices $\pm x$ and $\pm y$). Then $\xi'=\mu\xi$ for some $\mu\in\mathbb{R}$ and by Proposition \ref{hom} we get  $f(\xi)=f(\frac{1}{\mu}\xi')=\pm\frac{1}{\mu} f(\xi')\in\Pi'$.
Hence $f(\Pi) \subseteq \Pi'$.
The reverse inclusion $f^{-1}(\Pi') \subseteq \Pi$ is obtained in the same way thanks to the fact that $f$ is bijective. Therefore  $f(\Pi) = \Pi'$ and this completes the proof.	
\end{proof}

\section{The main results}

Now we are ready to state the main theorem of this note.
For $\dim{X} \geq 3$ we need  the fundamental theorem of projective geometry. For example the following version, see \cite{Frolicher,Faure,Havlicek}.

\begin{theorem}[Fundamental theorem of projective geometry]\label{projective}
	Let $X$ and $Y$ be real vector spaces of dimensions at least three.
	Let $\mathbb PX$ and $\mathbb PY$ be the sets of all one dimensional subspaces of $X$ and $Y$, respectively.
    Let $g \colon \mathbb PX\to\mathbb PY$ be a mapping such that
	\begin{itemize}
		\item[(i)] The image of $g$ is not contained in a projective line.
		\item[(ii)] $0\ne c\in \langle a,b\rangle, a\ne 0\ne b,$ implies $g(\langle c\rangle)\in \langle g(\langle a\rangle), g(\langle b\rangle)\rangle$.
	\end{itemize}
	Then there exists an injective linear mapping $A \colon X\to Y$ such that
	$$g(\langle x\rangle)=\langle Ax\rangle,\quad 0\ne x\in X.$$
	Moreover, $A$ is unique up to a non-zero scalar factor.
\end{theorem}

\begin{theorem}
Let $X$ and $Y$ be real normed spaces. Then a surjective mapping $f \colon X\to Y$ satisfies
\begin{equation}\label{main2}
\{\|f(x)+f(y)\|,\|f(x)-f(y)\|\}=\{\|x+y\|,\|x-y\|\}, \quad x,y\in X,
\end{equation}
if and only if $f$ is phase equivalent to a surjective linear isometry.
\end{theorem}

\begin{proof}
Let $\dim X\geq 3$. Because $f(\lambda x)=\pm\lambda f(x)$, $\lambda\in\mathbb{R}$ and $x\in X$, the mapping $\tilde{f}:\mathbb{P}X\to\mathbb{P}Y$, $\tilde{f}(\langle x\rangle)=\langle f(x)\rangle$ is well defined. From Proposition \ref{two} it follows that condition (ii) of Theorem \ref{projective} is satisfied. To see that (i) also holds suppose that the range $f(X)$ has dimension 2.
Then the range $f^{-1}(f(X))=X$ by Proposition \ref{two} also has dimension 2, a contradiction.
Therefore, by Theorem \ref{projective}  there exists an injective linear mapping $A \colon X\to Y$ such that
$$f(x)=\lambda(x) Ax.$$
By Lemma \ref{lin ind}(ii), $f$ preserves Birkhoff--James orthogonality (which is homogeneous), so $A$ preserves Birkhoff-James orthogonality as well.
Then $A$ is a scalar multiple of an isometry, see \cite{Blanco,Koldobsky}.
Write $A=\mu U$ for some $\mu \in \mathbb{R}$ and isometry $U \colon X \to Y$.
Since $f$ is norm preserving, $\vert \lambda(x)\mu \vert =1$, hence $\lambda(x)\mu = \pm 1$.
It remains to define $\sigma(x)=\lambda(x)\mu$ to get $f(x)=\sigma(x)Ux$.
Note that $U$ is surjective because $f$ is surjective.

If $\dim{X}=2$, then the result is proved in \cite{WB}.	

If $\dim X=1$, the result is proved in \cite[Proposition 2.2]{Huang2}, or we can give a short proof. Fix a unit vector $x_0\in X$. Let $\lambda\in\mathbb{R}$ and define $U(\lambda x_0)=\lambda f(x_0)$.
Then Proposition \ref{hom} implies $f(\lambda x_0)=\pm\lambda f(x_0)=\pm U(\lambda x_0)$. The mapping $U \colon X \to Y$ is linear and surjective. Lemma \ref{lemma preliminaries}(i) implies that $U$ is an isometry.

The proof is complete.
\end{proof}

\end{document}